\documentclass[12pt]{amsart}

\usepackage{amsmath,amssymb}
\usepackage{amsfonts}
\usepackage{amsthm}
\usepackage{latexsym}
\usepackage{graphicx}
\usepackage{epsfig}

\newtheorem{theorem}{Theorem}[section]
\newtheorem{lemma}[theorem]{Lemma}
\newtheorem{proposition}[theorem]{Proposition}
\newtheorem{definition}[theorem]{Definition}

\newtheorem{prop}{Proposition}[section]
\newtheorem{remark}[prop]{Remark}

\makeatletter \@addtoreset{equation}{section} \makeatother

\def\ppt{\frac{\partial}{\partial t}}

\def\FF{{\mathrm F}}
\def\RR{{\mathrm R}}
\def\WW{{\mathrm W}}

\def\EE{{\mathrm E}}

\def\Rc{{\mathrm {Rc}}}

\def\Rm{{\mathrm {Rm}}}

\begin{document}

\title{Mean Value Inequalities and Conditions to Extend Ricci Flow}

\author{Xiaodong Cao
}

\address{Department of Mathematics,
 Cornell University, Ithaca, NY 14853-4201}
\email{cao@math.cornell.edu, hungtran@math.cornell.edu}

\author{Hung Tran}


\renewcommand{\subjclassname}{%
  \textup{2000} Mathematics Subject Classification}

\date{\today}

\begin{abstract} This paper concerns  conditions related to the first finite singularity time of a Ricci flow solution on a closed manifold. In particular, we provide a systematic approach to the mean value inequality method, suggested by N. Le  \cite{LN10} and F. He  \cite{FH12}. We also display a close connection between this method and time slice analysis as in \cite{wangb12}. As an application, we prove several inequalities for a Ricci flow solution on a manifold with nonnegative isotropic curvature.
\end{abstract}
\maketitle

\markboth{Xiaodong Cao and Hung Tran} {Mean Value Inequalities and Conditions to Extend Ricci Flow}
\section{\textbf{Introduction}}
Let $M^n$ be a n-dimensional closed manifold and $g(t)$, $0\leq t<T$, be a one-parameter family of smooth Riemannian metrics on M satisfying $$\ppt g(t)=-2\text{Rc}(t),$$ then $(M,g(t))$ is said to be a solution to the Ricci flow,  which was first introduced by R. Hamilton in \cite{H3}. The Ricci flow has been studied extensively, particularly in the last decade thanks to the seminal contribution by G. Perelman \cite{perelman1}. As a weakly parabolic system, the Ricci flow can develop finite-time singularities. We say that $(M, g(t))$, $t\in[0,T)$, is a maximal solution if it becomes singular at time $T$. An intriguing but still open problem concerns the precise behavior of the curvature tensor as the flow approaches the finite singular time.\\

It was first shown by Hamilton that the norm of the Riemannian curvature tensor $|\Rm|$ must blow up approaching the first finite singular time \cite{H3}. More recently, by using an application of the non-collapsing result of G. Perelman \cite{perelman1}, N. Sesum was able to prove that if the norm of Ricci curvature $|\Rc|$ is bounded then the flow can be extended \cite{sesum051}. Since then the obvious generalized problem of whether the scalar curvature must behave similarly has received extensive attention. The question is still open but considerable progress has been made: the Type I case is resolved by J. Ender, R. Muller and P. Topping \cite{emt10},  also independently by Q. Zhang and the first author in \cite{cz10, cao11},  while the K\"{a}hler case is solved by Z. Zhang \cite{zhangz10}. There are various other relevant results such as estimates relating the scalar curvature and the Weyl tensor \cite{cao11}, comparable growth rates of different components of the curvature tensor \cite{wangb12}, \cite{wangb08}, and integral conditions by N. Le and N. Sesum \cite{ls12}.\\

Following the mean value inequality trick of Le \cite{LN10} for the mean curvature flow, F. He developed a logarithmic-improvement condition for the Ricci flow \cite{FH12}. Our contribution is to provide a more systematic treatment of the mean value inequality method  and to find a close connection to the time slice analysis method suggested by B. Wang \cite{wangb12}. Then we apply our analysis to a particular context of Ricci flow with a uniform-growth condition defined below.\\

Throughout the rest of this paper, $M$ is a closed smooth  Riemannian manifold and $T$ is some positive finite time. We will use the following notation:
$$ Q(t)=\sup_{M\times \{t\}}|\text{Rm}|,~P(t)=\sup_{M\times \{t\}}|\text{Rc}|,~O(t)=\sup_{M\times \{t\}}|\RR|.$$
Our first theorem gives a logarithmic-improvement condition relating the Ricci curvature and the Riemannian curvature tensor.

\begin{theorem}\label{mt1}
Let $(M,g(t))$, $t\in [0,T)$, be a Ricci flow solution on  M. If for some $0\leq p\leq 1$, we have $$\int_{0}^{T}\frac{P(t)}{(\ln(1+Q(t)))^{p}}dt< \infty,$$  then the solution can be extended past time T.
\end{theorem}

Since we are interested in the behavior of the scalar curvature at a singular time, this motivates the following definition.

\begin{definition} 
\label{unigrowth}
A Ricci flow solution on a closed manifold is said to satisfy the \textbf{uniform-growth} condition if it develops a singularity in finite time, and any singularity model obtained by parabolic rescaling at the scale of the maximum curvature tensor must has non-flat scalar curvature.  
\end{definition}

Under the Ricci flow, the uniform-growth condition  generalizes both Type I and (non-flat) nonnegative isotropic curvature (NIC) conditions. Combining the above mean value inequality method with the uniform-growth condition  yields the following logarithmic-improvement result.

\begin{theorem}\label{mt2}
Let $(M,g(t))$, $t\in [0,T)$, be a Ricci flow solution satisfying the uniform-growth condition on  M. If for some $0\leq p\leq 1$, we have
\begin{equation}
\int_{0}^{T}\int_{M}\frac{|\RR|^{n/2+1}}{(\ln{(1+|\RR|)})^{p}} d\mu dt< \infty , 
\end{equation}
then  the solution can be extended past time T.  
\end{theorem}

The paper is organized as follows. In Section 2, we discuss mean value inequalities and provide the proof of Theorem \ref{mt1}. Section 3 displays a close connection to the time-slice analysis and thus gives another proof of the above result as well as  some independent estimates. In Section 4 we apply our method to the context of nonnegative isotropic curvature and its generalization.

\section{Mean Value Inequalities}


In this section, we first  generalize a trick in \cite{LN10} that connects a mean value inequality with the finiteness of a function.
\begin{lemma}
\label{namtrick}
Let $f,G:[0,T)\rightarrow [0,\infty)$ be continuous functions and $\psi: [0,\infty)\rightarrow [0,\infty)$ be a non-decreasing function such that
\begin{equation}
\int_{1}^{\infty}\frac{1}{\psi(s)}ds=\infty. 
\end{equation}
If there is a mean value inequality of the form
\begin{equation}
f(t)\leq  C_{1} \int_{0}^{t}\psi(f(s))G(s)ds+C_{2}=h(t)
\end{equation}
and $\int_{0}^{T}G(t)dt<\infty$, then $\limsup_{t\rightarrow T}f(t)<\infty$.
\end{lemma}

\begin{proof}
For any $T_{0}<T$,
\begin{align*}
\int_{0}^{T_{0}}C_{1}G(t)dt&=\int_{0}^{T_{0}}\frac{1}{\psi(f(t))}C_{1}\psi(f(t))G(t)dt\\
&=\int_{h(0)}^{h(T_{0})}\frac{1}{\psi(f(h^{-1}(s)))}ds \text{ (let $s=h(t)$, $ds=h'(t)dt$)}\\
& \geq \int_{h(0)}^{h(T_{0})}\frac{1}{\psi(s)}ds.
\end{align*}
The last inequality is because  of $f(t)\leq h(t)$. If $\int_{0}^{T}C_{1}G(t)dt<\infty$, then by the choice of $\psi$, $h(T_{0})\leq C<\infty$. Now by the mean value inequality, $f(T_{0})\leq h(T_{0})\leq C$. Since $T_{0}$ is arbitrary, $\sup_{[0,T)}f\leq C<\infty$.
\end{proof}
In the next several lemmas, we will establish a mean value inequality connect $Q(t)$ and $P(t)$. We first need the following doubling time estimate.
\begin{lemma}\cite[Lemma 6.1]{chowluni}
\label{doublingtime}
If $(M,g(t))$ is a Ricci flow on a closed manifold then for all $t\in[0,\frac{1}{16Q_{0}})$, 
\begin{equation}
Q(t)\leq 2Q(0).
 \end{equation} 
\end{lemma}

We also need the following definition, for detailed discussion, see \cite{perelman1}.


\begin{definition}
A  Ricci flow solution $(M, g(t))$, $t\in (0, T)$,    is said $\kappa$-noncollapsed  {\em(}on all scales{\em)},
if   $\forall g(t)$, every metric ball $B$ of radius $r$, with $|Rm|(x) \le r^{-2}$ for every $x\in B$, has volume at least $\kappa r^n$.  
\end{definition}

We are now ready to state our first key technical lemma.

\begin{lemma}
\label{subRmRc}
Let $\Sigma(M,\kappa,C_{0})=\{g(t)| t\in[0,1], \text{$g(t)$ is $\kappa$-noncollapsed}, ~ Q(0)\leq C_{0}\}$ be a set of complete Ricci flow solutions on  M. Then there exists a constant $C=C(n,\kappa,C_{0})$ such that for any $g(t)\in \Sigma$,
\begin{equation}
\sup_{[0,1]}Q(t)\leq C\int_{0}^{1}Q(t)P(t)dt+32C_{0}.
\end{equation}
\end{lemma}

\begin{proof}
The proof is by contradiction. Suppose that the statement is false then there exists a sequence of $g_{i}(t)\in \Sigma$ and $a_{i}\rightarrow \infty$ such that 
\begin{equation*}
\sup_{[0,1]}Q_{i}(t)\geq a_{i}\int_{0}^{1}Q_{i}(t)P_{i}(t)dt+32C_{0}.
\end{equation*}
Let $Q_{i}=\sup_{[0,1]}Q_{i}(t)$ then we can find $(x_{i},t_{i})$ such that $Q_{i}$ is attained. Since $Q_{i}>32 C_{0}$ there exists $t_{i0}$ being the first time backward such that $Q_{i}(t_{i0})=\frac{1}{2}Q_{i}$. Consequently, for $t \in [t_{i0},t_{i}]$, $32C_{0}<Q_{i}<2Q_{i}(t)$, $Q_{i}(t_{i0})>16C_{0}$ and by Lemma \ref{doublingtime}, $t_{i0}>\frac{1}{16C_{0}}$. \\

{\bf Claim}: There exists a constant $\epsilon_{0}=\epsilon_{0}(n,\kappa)$ such that the following holds: for any $t_{0}>0$, $D\geq \max \{1/t_{0},\max_{[0,t_{0}]}Q\}$, let $t_{1}>t_{0}$ be the first time, if exists, such that $Q(t_{1})=D$, and $t_{2}>t_{1}$ be the first time, if exists, such that $|\ln(Q(t_{2})/Q(t_{1}))|=\ln{2}$, then      
\begin{equation*}
\int_{t_{1}}^{t_{2}}P(t)dt> \epsilon_{0}.
\end{equation*}

{\it Proof of claim:} This is essentially just a restatement of \cite[Lemma 3.2]{wangb12}. If there are no such $t_{1},t_{2}$, the statement is vacuously true. If they exist then we dilate the solution by $\tilde{g}(t)=Dg(t_{1}+t/D)$ then $\tilde{g}(t)$ satisfies the condition of the aforementioned result and the claim follows after rescaling back.\\

Applying the claim above yields
\begin{equation}
\label{subRmRc1}
\int_{t_{i0}}^{t_{i}}P_{i}(t)dt>\epsilon_{0}.
\end{equation}
Thus, 
\begin{equation}
\label{subRmRc2}
{Q}_{i}\geq 32C_{0}+a_{i}\int_{t_{i0}}^{t_{i}}Q_{i}(t)P_{i}(t) dt \geq 32C_{0}+a_{i}16C_{0}\epsilon_{0}.
\end{equation}
On the other hand, 
\begin{align*}
{Q}_{i}\int_{t_{i0}}^{t_{i}}P_{i}(t)dt &\leq 2\int_{t_{i0}}^{t_{i}}Q_{i}(t)P_{i}(t)dt\\
&\leq 2\int_{0}^{1}Q_{i}(t)P_{i}(t)dt\\
&\leq 2\frac{{Q}_{i}-32C_{0}}{a_{i}},
\end{align*} hence
$$\int_{t_{i0}}^{t_{i}}P_{i}(t)dt \leq \frac{2}{a_{i}}\frac{{Q}_{i}-32C_{0}}{{Q_{i}}}\rightarrow 0 ,$$
the last limit follows from (\ref{subRmRc2}) and $a_{i}\rightarrow \infty$. This is in contradiction with (\ref{subRmRc1}), so the lemma follows.
\end{proof}


We are now in the position to state our mean value inequality. 

\begin{proposition}
\label{RmRc}
Let $(M,g(t))$, $0\leq t<T$, be a Ricci flow solution on $M$ and $Q(t)=\sup_{M\times\{t\}}|\text{Rm}|$. There exist constants 
\begin{align*}
C_{0}&=C_{0}(n,\kappa,Q(0)),\\
C_{1}&=32Q_{0},
\end{align*}
such that
\begin{equation}
\sup_{[0,t]}Q\leq C_{0}\int_{0}^{t}Q(u)P(u)du+C_{1}.
\end{equation} 
\end{proposition}

\begin{proof}
For $t\in[0,\frac{1}{16Q(0)})$ the statement is true by Lemma \ref{doublingtime}. For any $t\in [\frac{1}{16Q(0)},T)$ define 
\begin{align*}
\tilde{g}(s) &=\frac{1}{t}g(ts), \text{ $s\in[0,1]$},\\
\tilde{Q}(0) &=tQ(0).
\end{align*}
Since the noncollapsing constant is a scaling invariant, applying Lemma \ref{subRmRc} yields
\begin{align*}
\sup_{[0,1]}\tilde{Q} &\leq C_{0}\int_{0}^{1}\tilde{Q}(s)\tilde{P}(s)ds+32tQ_{0},\\
\sup_{[0,t]}tQ &\leq C_{0}t\int_{0}^{t}Q(u)P(u)du+32tQ(0) \text{ ($u=ts$)}, \\
\sup_{[0,t]}Q &\leq C_{0}\int_{0}^{t}Q(u)^{2}du+32Q(0).
\end{align*}
\end{proof}

Now we can finish the proof of Theorem \ref{mt1}.
\begin{proof}({\bf Theorem \ref{mt1}}) First observe that if T is the first singular time then 
\[\lim_{t\rightarrow T} Q(t)=\infty \]
by \cite{H3}. Now applying Lemma \ref{namtrick} with the function $\psi(s)=s\ln(1+s)^{p}$, $0\leq p\leq 1$ (it is easy to check that it is nondecreasing and $\int_{1}^{\infty}\frac{1}{\psi(s)}ds=\infty$) and Proposition \ref{RmRc} yields the result.
\end{proof}

\section{Time Slice Approach}

In the last section, the essential ingredient to obtain the mean value inequality relating $Q(t)$ and $P(t)$ is the estimate in Lemma \ref{subRmRc}. That estimate essentially states that when the curvature double the integral of the maximum norm of the Ricci tensor be bounded below by some universal constant. It turns out that using the time slice analysis, we can deduct similar results in a slightly different manner. To be more precise, the logarithmic quantity and $ln(\int_{0}^{T}P(t)dt)$ blow up together at the first singular time.  We shall also derive some other results which might be of independent interest. 

First let's fix our notation. For a Ricci flow solution developing a finite time singularity, let $s_{i}$ be the first time such that $Q(s_{i})=2^{i+4}Q(0)$. Then by Lemma \ref{doublingtime}, 
\begin{equation}
\label{doubletime}
s_{i+1}\geq s_{i}+\frac{1}{16Q(s_{i})}=s_{i}+\frac{1}{8Q(s_{i+1})}.
\end{equation}

\begin{lemma}
\label{expRc}
Let $(M,g(t))$, $t\in[0,T)$, be a maximal $\kappa$-noncollapsed Ricci flow solution on  $M$ . Then 
\begin{equation}
\sup_{[0,t]}Q(s) \leq 2^{\frac{1}{\epsilon_{0}}\int_{0}^{t}P(s)ds+1}16Q(0),
\end{equation}
where $\epsilon_{0}$ is the constant from the claim of Lemma \ref{subRmRc}.
\end{lemma}

\begin{proof}
The result can be deducted directly from \cite[Theorem 3.1]{wangb12}. For completeness, we provide a proof here. 
From the claim in Lemma \ref{subRmRc},  we have
$$\int_{s_{i}}^{s_{i+1}}P(t)dt\geq \epsilon_{0}.$$
Let N be the largest interger such that $s_{N}\leq t$ then 
$$N\epsilon_{0}  \leq \int_{s_{0}}^{s_{N}}P(s)ds\leq \int_{0}^{t}P(s)ds,$$ hence
$$N \leq \frac{1}{\epsilon_{0}}\int_{0}^{t}P(s)ds.$$
Thus it follows that \[\sup_{[0,t]}Q(s)\leq 2^{N+1}16Q(0)\leq 2^{\frac{1}{\epsilon_{0}}\int_{0}^{t}P(s)ds+1}16Q(0).\]
\end{proof}

Next we derive a mean value type inequality using the time slice argument.

\begin{theorem}
\label{sec3th1}
Let $(M,g(t))$, $t\in[0,T)$, be a maximal $\kappa$-noncollapsed Ricci flow solution on  $M$. Furthermore, let  \[G(u)=\ln(16Q(0))+2\ln{2}+\frac{\ln{2}}{\epsilon_{0}}\int_{0}^{u}P(s)ds.\] 
Then for $0\leq p\leq 1$, we have
\begin{equation}
\ln(G(t)) \leq C_1 \int_{0}^{t}\frac{P(s)}{(\ln(1+Q(s)))^{p}}ds+C_{2},
\end{equation}
where $C_{1}>0$ only depends on $\epsilon_{0}$, $C_{2}>0$ depends on $\epsilon_{0}$ and $Q(0)$.
\end{theorem}
\begin{proof}
First, without loss of generality, let $Q=\sup_{[0,t]}Q(s)>2$ and observe that for $0\leq p\leq 1$, 
$$(\ln(1+Q(s)))^{p}\leq \ln(1+Q(s))\leq \ln(1+Q).$$
Applying Lemma \ref{expRc}, 
\begin{align*}
1+Q&\leq 2^{\frac{1}{\epsilon_{0}}\int_{0}^{t}P(s)ds+2}16Q(0),\\
\ln(1+Q) &\leq \ln(16Q(0))+2\ln{2}+\frac{\ln{2}}{\epsilon_{0}}\int_{0}^{t}P(s)ds.
\end{align*} 
Since $G(u)=\ln(16Q(0))+2\ln{2}+\ln{2}\int_{0}^{u}P(s)ds$, we have 
$$G'(s)=\frac{\ln{2}}{\epsilon_{0}}P(s)>0,$$ 
and $$G(s) \geq (\ln(1+Q(s)))^p.$$ 
Therefore, 
\begin{align*}
\frac{\ln{2}}{\epsilon_{0}} \int_{0}^{t}\frac{P(s)}{(\ln(1+Q(s)))^{p}}ds& \geq \int_{0}^{t}\frac{G'(s)}{G(s)}ds\\
& =\ln G(t)-\ln G(0).
\end{align*}
The statement now follows immediately.
\end{proof}

\begin{remark}
Theorem \ref{mt1} now follows from Theorem \ref{sec3th1} and the fact that $\int_0^T P(s)ds$ needs to blow up at the first singular time $T$ {\em (for example, see \cite{wangb12} or \cite{FH12})}.
\end{remark}
Next we apply the same method to a slightly different setting.

\begin{lemma}
\label{meanint}
 Let $(M,g(t))$, $t\in[0,T)$, be a maximal $\kappa$-noncollapsed Ricci flow solution on  $M$. Then there exists a constant $C=C(Q(0),\kappa)$, such that 
\begin{equation}
Q(s_{i+1})\leq C\int_{s_{i}}^{s_{i+1}}\int_{M}|\Rm|^{\frac{n}{2}+2}d\mu_{g(s)}ds,
\end{equation}
and thus
\begin{equation}
\frac{1}{C}\leq \int_{s_{i}}^{s_{i+1}}\int_{M}|\Rm|^{\frac{n}{2}+1}d\mu_{g(s)}ds.
\end{equation}
\end{lemma}

\begin{proof}
Suppose that the statement is false then as $j\rightarrow \infty$, there exist $s_{i_{j}}\rightarrow T$ and $a_{j}\rightarrow \infty$, such that 
\begin{equation*}
a_{j}\int_{s_{i_{j}}}^{s_{i_{j}+1}}\int_{M}|\Rm|^{n/2+2}d\mu_{g(s)}ds \leq Q(s_{i_{j}+1}).
\end{equation*}
Therefore, we can choose a blow-up sequence $(x_{j},s_{i_{j}}+1)$ (in the sense of \cite[Theorem 8.4]{chowluni}) and rescale the metric by
$$g_{j}(t)=Q(s_{i_{j}+1})g(s_{i_{j}+1}+\frac{t}{Q(s_{i_{j}+1})}).$$  
By the Cheeger-Gromov-Hamilton compactness theorem and Perelman's non-collapsing result (for more details, see \cite[Chapter 8]{chowluni}), we obtain a singularity model $(M_{\infty},g_{\infty}(s),x_{\infty})$ with $|\Rm_{\infty}(x_{\infty},0)|=1.$

On the other hand, 
\begin{align*}
\int_{-1/8}^{0}\int_{M}|\Rm(g_{j}(t))|^{\frac{n}{2}+2}d\mu_{g_{j}(t)}dt&=\frac{1}{Q(s_{i_{j}+1})}\int_{s_{i_{j}+1}-\frac{1}{8Q(s_{i_{j}+1})}}^{s_{i_{j}+1}}\int_{M}|\Rm(g(s)|^{\frac{n}{2}+2}d\mu_{g(s)}ds\\
&\leq \frac{1}{Q(s_{i_{j}+1})}\int_{s_{i_{j}}}^{s_{i_{j}+1}}\int_{M}|\Rm(g(s)|^{\frac{n}{2}+2}d\mu_{g(s)}ds\\
& \leq \frac{1}{a_{j}} \rightarrow 0,
\end{align*}
here (\ref{doubletime}) is used in the first inequality. However, by the dominating convergence theorem, the limit solution is flat, this is a contradiction.

The second statement follows from the first immediately.
\end{proof}
Note that Lemma \ref{meanint} involves a time slice estimate similar in the spirit of the claim in Lemma \ref{subRmRc} and, thus, applying the same method as before yields the following results. The proofs are omitted as they are almost identical to those of Lemma \ref{expRc} and Theorem \ref{sec3th1}.
\begin{proposition}
Let $(M,g(t))$, $t\in[0,T)$, be a maximal $\kappa$-noncollapsed Ricci flow solution on $M$.  Then
\begin{equation}
\sup_{[0,t]}Q(s) \leq 2^{C\int_{0}^{t}\int_{M}|\Rm|^{\frac{n}{2}+1}d\mu_{g(s)}ds+1}16Q(0).
\end{equation}
\end{proposition}

\begin{theorem}
\label{sec3th2}
Let $(M,g(t))$, $t\in[0,T)$, be a maximal $\kappa$-noncollapsed Ricci flow solution on $M$.  Let
\[G(u)=\ln(16Q(0))+2\ln{2}+C\ln{2}\int_{0}^{u}\int_{M}|\Rm|^{\frac{n}{2}+1}d\mu_{g(s)}ds.\]
Then for $0\leq p\leq 1$, we have 
\begin{equation}
\ln(G(t)) \leq C_1 \int_{0}^{t}\int_{M}\frac{|\Rm|^{\frac{n}{2}+1}}{(\ln(1+\Rm))^{p}}d\mu_{g(s)}ds+C_{2},
\end{equation}
where $C_{1}>0$ and $C_{2}$ only depend on  $\kappa$ and $Q(0)$.
\end{theorem}

\begin{remark} It is shown in \cite{wangb08} that the function $G(t)$ must blows up as $t$ approaching the first singular time. Therefore, Theorem \ref{sec3th2} implies \cite[Theorem 1.6]{FH12}.
\end{remark}

\section{Nonnegative Isotropic Curvature Condition}

The notion of nonnegative isotropic curvature was first introduced by M. Micallef and J. D. Moore in \cite{mm88}. A Riemannian manifold M of dimension $n\geq 4$ is said to have nonnegative isotropic curvature if for every orthonormal $4$-frame $\{e_1, e_2, e_3, e_4\}$, that
\[\RR_{1313}+\RR_{1414}+\RR_{2323}+\RR_{2424}-2\RR_{1234}\geq 0.\]
The positive condition is defined similarly by replacing the above   with a strict inequality. The isotropic curvature is also related to  complex sectional curvatures  described as follows. For each $p\in M$, let $T_{p}^{C}M=T_{p}M\otimes_{\mathbb{R}} \mathbb{C}$, then the Riemannian metric g extends naturally to a complex bilinear form $$g:T_{p}^{C}M \times T_{p}^{C}M \rightarrow \mathbb{C},$$ and so is the Riemannian curvature tensor  $\Rm$ to a complex multilinear form $$\Rm: T_{p}^{C}M \times T_{p}^{C}M \times T_{p}^{C}M \times T_{p}^{C}M\rightarrow \mathbb{C}.$$ Then $M$ has nonnegative isotropic curvature if and only if, 
 \[\Rm(\theta, \eta,\overline{\theta},\overline{\eta})\geq 0\]
for all  (complex) vectors $\theta$, $\eta$ satisfying $g(\theta,\theta)=g(\eta,\eta)=g(\theta,\eta)=0$ (such a plane spanned by $\theta$ and $\eta$ is called an isotropic plane, for more details, see \cite{brendlebook10}). Furthermore, this nonnegative isotropic curvature condition is implied by several other commonly used curvature
conditions, such as  nonnegative curvature operator or point-wise
$\frac14$-pinched sectional curvature conditions, and it implies nonnegative scalar curvature. For more details, please check, for example,  \cite{mm88} or \cite{brendlebook10}. \\

 Another interesting and relevant fact is that this condition is preserved along the Ricci flow. In dimension $4$, it was proved by Hamilton \cite{HPIC}; higher dimension analog was extended by S. Brendle and R. Schoen \cite{bs091} and
also by H. Nguyen \cite{nguyen10} independently. Using minimal surface technique,  Micallef and Moore \cite{mm88} showed that
 any compact, simply connected manifold with
 positive isotropic curvature is homeomorphic to
 $S\sp n$. By utilizing the Ricci flow and the aforementioned perseverance, Brendle and
Schoen further proved the Differentiable Sphere theorem, which has
been a long time conjecture since the (topological)
$\frac14$-pinched Sphere theorem was proved by M. Berger
\cite{berger60} and W. Klingenberg \cite{kl61} around 60's. More
precisely, Brendle and Schoen showed that any compact Riemannian
manifold with point-wise $\frac14$-pinchied
sectional curvature is diffeomorphic to a spherical space form \cite{bs091}.\\

In this section, we   apply our analysis to the context of non-flat manifolds with nonnegative isotropic curvature or, slightly more general, satisfying the uniform-growth assumption as in Definition \ref{unigrowth}. Let's  first recall the definition of flag curvature and Berger's Lemma.

\begin{definition} Given a unit vector e, the flag curvature on the direction e is a symmetric bilinear form on $V_{e}=e^{\perp}$ (the perpendicular compliment of e in $V=R^{n}$) given by $\RR_{e}(X,X)=\Rm({e},X,{e},X)$ for any $X \in V_{e}$.
\end{definition}

We further define $\rho_{e}=\sup_{|X|=|Y|=1,<X,Y>=0} {(R_{e}(X,X)-R_{e}(Y,Y))}$ and $\rho=\sup_{e}{\rho_{e}}$.

\begin{remark} It is clear that $\rho$ is no more than the difference between the maximum and minimum of sectional curvatures at each point.
\end{remark}
\begin{lemma}[Berger \cite{berger60b}]
\label{Berger} 
For orthonormal vectors U, V, X, W in $T_{p}M$, we have\\
{\bf a}{\em)} $|\Rm(U,V,U,W)|\leq \frac{1}{2}\rho_{U}$,\\
{\bf b}{\em)} $|\Rm(U,V,X,W)|\leq \frac {1}{6}\rho_{U+X}+\frac{1}{6}\rho_{U-X}+ \frac {1}{6}\rho_{U+W}+\frac {1}{6}\rho_{U-W}\leq \frac{2}{3}\rho$.\\
\end{lemma}
The Weitzenb\"{o}ck operator $\FF$ is defined as
$$\FF=\Rc \circ g -2\Rm= \frac{(n-2)\RR}{n(n-1)}g\circ
g+\frac{n-4}{n-2} \EE\circ g -\WW.$$
It is well-known that in dimension four, NIC is equivalent to the nonnegativity of $\FF$ (see, for example,
\cite{mm88, mw93, noronha97}). Furthermore, the space of bi-vectors $\Lambda^{2}$ can be decomposed into $\Lambda^{2}_{+}$ and $\Lambda^{2}_{-}$ by the Hodge operator $\ast $. In particular, the Weyl tensor, considered as an operator on 2-forms,  is structurally represented as
\[\WW=\WW_{+}+\WW_{+}.\] 
As a consequence, $\FF \geq 0$ is equivalent to, for $I_{\pm}$ the identity operators on $\Lambda^{2}_{\pm}$,
\[\frac{\RR}{6}I_{\pm}-W_{\pm} \geq 0.\] 
We need the following lemma.

\begin{lemma} 
\label{picestimate}
Let $(M,g)$ be a manifold with NIC then the followings holds.\\
{\bf a}{\em)} If $n=4$, $ |W| \leq \frac{2}{\sqrt{3}}\RR.$\\
{\bf b}{\em)} If $n>4$,  $ |\Rm|  \leq c(n)\RR.$
\end{lemma}

Part b) is well-known to many experts, for example, see \cite{seshadri10} or  \cite[Prop. 7.3]{brendlebook10}. We provide a proof here for the sake of  completeness. 

\begin{proof} {\bf a})
Let $\lambda_{1}\leq \lambda_{2}\leq \lambda_{3}$ be eigenvalues of $W_{+}$ then 
\begin{align*}
|\WW^{+}|^2 &=4\sum_{i=1}^{3}\lambda_{i}^2,\\
0&=\sum_{i=1}^{3}\lambda_{i},\\
-\frac{\RR}{3} &\leq \lambda_{i} \leq \frac{\RR}{6},
\end{align*}
while noticing that, for a $(4,0)$-tensor, the tensor norm is 4 times the operator norm (sum of squared eigenvalues).
We would like to maximize the function $|\WW^{+}|^2 =4\sum_{i=1}^{3}\lambda_{i}^2$ on the region identified by the plane $\sum_{i=1}^{3}\lambda_{i}=0$ bounded by the the box $-\frac{\RR}{3}\leq \lambda_{i} \leq \frac{\RR}{6}$. Since the region is compact, the function attains its maximum. \\
If $\lambda_{1}>-\frac{\RR}{3}$ then we can always increase the function by decreasing $\lambda_{1}$ and increasing either $\lambda_{2}$ or $\lambda_{3}$. Thus $|\WW_{+}|$ attains its maximum when $\lambda_{1}=-\frac{\RR}{3}$ and $\lambda_{2}=\lambda_{3}=\frac{\RR}{6}$. Clearly, the argument holds for $\WW_{-}$ and the result follows.

{\bf b}) If $n>4$, then we have
\begin{align*}
&\RR_{ikik}+\RR_{ilil}+\RR_{jkjk}+\RR_{jljl} \geq 0,\\
&\RR_{ii}+\RR_{jj}\geq 2\RR_{ijij},\\
&(n-4)\RR_{ii}+\RR\geq 0.
\end{align*}
Thus,
\begin{align*}
&\RR_{ii}\geq -\frac{\RR}{n-4},\\
&\RR_{ii}=\RR-\Sigma_{j\neq i}\RR_{jj}\leq \RR+(n-1)\frac{\RR}{n-4}=c_{1}\RR, \\
&\RR_{ijij}\leq \frac{1}{2}(\RR_{ii}+\RR_{jj})\leq c_{1}\RR,\\
&\RR_{ijij}\geq -3c_{1}\RR,
\end{align*}
Now by Lemma \ref{Berger}, 
\begin{align}
|\RR_{ijik}|\leq 2c_{1}\RR,\\
|\RR_{ijkl}|\leq \frac{8}{3}c_{1}\RR. 
\end{align}
Thus, there exists a constant $c(n)$ such that
\begin{equation*}
|\Rm|\leq c(n) |\RR|.
\end{equation*} 
\end{proof}

A direct consequence of the above lemma is the following proposition.

\begin{proposition}Let $(M,g(t))$, $t\in [0,T)$, be a maximal Ricci flow solution with NIC, then there exists $c=c(n,g(0))$ such that $|\text{Rm}|\leq c\RR$ along the flow.
\end{proposition}
\begin{proof}
If $n>4$ then the result follows from part b) of  Lemma \ref{picestimate}.  \\
If $n=4$, then by the pinching estimate of \cite{cao11},
\begin{equation*}
\frac{|\mathring{\text{Rc}}|}{\RR}\leq c_{1}(n,g(0))+c_{2}(n)\sup_{M \times [0,T)}\sqrt{\frac{|W|}{R}}\leq c_{1}+c_{2}\sqrt{\frac{2}{\sqrt{3}}}.
\end{equation*} 
Furthermore, $|\text{Rm}|^2=|W|^2+\frac{\RR^2}{6}+2|\mathring{\text{Rc}}|^2$, the result follows.
\end{proof}
\begin{remark} One easy consequence is that a non-flat Ricci flow solution on a closed manifold with NIC satisfies the uniform-growth condition as in Definition \ref{unigrowth}.
\end{remark}

\begin{theorem}
\label{integralcond}
Let $(M,g(t))$, $t\in[0,T)$, be a Ricci flow solution satisfying the uniform-growth condition. If either
\[ \int_{M}|\RR|^{\alpha} d\mu_{g(t)}\leq \infty, \text{ for some $\alpha> n/2$},\]
or 
\[\int_{0}^{T}\int_{M}|\RR|^{\alpha} d\mu_{g(t)} dt\leq \infty \text{ for some $\alpha\geq \frac{n}{2}+1$},\]
then the solution can be extended past time T.
\end{theorem}

\begin{proof}

First we observe that, by Holder inequality, for the second condition, it suffices to prove the case $\alpha=\frac{n}{2}+1$.

The proof is by a contradiction argument. Suppose the flow develops a singularity at time $T$ then we can choose a blow-up sequence $(x_{i},t_{i})$ (in the sense of \cite[Theorem 8.4]{chowluni}), and rescale the metric by $g_{i}(s)=Q_{i}g(t_{i}+\frac{s}{Q_{i}})$. By the Cheeger-Gromov-Hamilton compactness theorem and Perelman's non-collapsing result (for more details, see \cite[Chapter 8]{chowluni}), we obtain a singularity model $(M_{\infty},g_{\infty}(s),x_{\infty})$ with
\begin{equation}
\label{nonflatlimit}
|\Rm_{\infty}(x_{\infty},0)|=1.
\end{equation} 
Using the scaling invariant of $\RR$ and the assumptions above we calculate:
\begin{align*}
\int_{M}|\RR(g_{i}(.))|^{\alpha}d\mu_{g_{i}(.)}&=\int_{M}Q_{i}^{-\alpha}|\RR(g(.)|^{\alpha}Q_{i}^{n/2}d\mu_{g(.)}\\
&=Q_{i}^{\frac{n}{2}-\alpha}\int_{M}|\RR|^{\alpha} d\mu_{g(.)}\rightarrow 0 \text{ as $i \rightarrow \infty$}.
\end{align*}
In the second case, we have:
\begin{align*}
\int_{-1}^{0}\int_{M}|\RR(g_{i}(s))|^{\frac{n}{2}+1}d\mu_{g_{i}(s)}ds&=\int_{t_{i}-\frac{1}{Q_{i}}}^{t_{i}}\int_{M}Q_{i}^{-\frac{n}{2}-1}|\RR(g(t)|^{\alpha}Q_{i}^{n/2}d\mu_{g(t)}Q_{i}dt\\
&=\int_{t_{i}-\frac{1}{Q_{i}}}^{t_{i}}\int_{M}|\RR(g(t)|^{\alpha}d\mu_{g(t)}dt \rightarrow 0 \text{ as $i \rightarrow \infty$}.
\end{align*}
By the dominating convergence theorem, the singularity model  $(M_{\infty},g_{\infty}(s),x_{\infty})$ is scalar flat, which is a contradiction to our uniform-growth condition.
\end{proof}

Applying Lemma \ref{namtrick} in this context, we obtain the following lemma. 
\begin{lemma}
\label{picnamtrick}
Let $(M,g(t))$, $t\in[0,T)$,  be a Ricci flow solution satisfying the uniform-growth condition. Suppose $\psi: (0,\infty)\rightarrow (0,\infty)$ is a nondecreasing function such that
\begin{equation}
\int_{1}^{\infty}\frac{1}{\psi(s)}ds=\infty.
\end{equation}
If there is a mean value inequality of the form
\begin{equation}
O(t)\leq \int_{0}^{t}C_{1}\psi(f(s))G(s)ds+C_{2}=h(t),
\end{equation}
and $\int_{0}^{T}G(t)dt<\infty$, then the solution can be extended past time T.
\end{lemma}
\begin{proof}
 First observe that if T is a first singular time then 
\[\lim_{t\rightarrow T} Q(t)=\infty \]
by \cite{H3}. The uniform-growth condition implies that the curvature tensor and the scalar curvature blow up together. Applying Lemma \ref{namtrick} we obtain a contradiction, hence the result holds. 
\end{proof}

We are ready to state a mean value inequality.

\begin{lemma} 
\label{picmvi}
Let $(M,g(t))$, $t\in[0,T)$, be a maximal Ricci flow solution satisfying the uniform-growth condition. Then the following mean value inequality holds: there exists $C_{1}=C_{1}(n,g(0))$ and  $C_{0}$ such that,
\begin{equation}
\sup_{[0,t]}O(t)\leq C_{0}\int_{0}^{t}\int_{M}|\RR(g(t))|^{n/2+2}d\mu_{g(t)}dt+C_{1}
\end{equation}
for all $t<T$.  
\end{lemma}

\begin{proof}
First we observe that there is a constant $c_{0}(n)$ such that $|\RR|(x,t)\leq c_{0}|\text{Rm}|(x,t)$. Also by Lemma \ref{doublingtime}, if $t\leq \frac{1}{16Q_{0}}$ then 
\begin{equation}
\label{picmvi1}
O(t)\leq c_{0}Q(t)\leq 2c_{0}Q(0). 
\end{equation}
Let \[C_{1}=2c_{0}Q(0).\]
Now suppose the statement is false then there exist sequences $t_{i}\rightarrow T$ and $a_{i}\rightarrow \infty$ such that 
\begin{equation*}
a_{i}\int_{0}^{t_{i}}\int_{M}|\RR|^{n/2+2}d\mu_{g(s)}ds+2c_{0}Q(0)\leq \sup_{[0,t_{i}]}O(t)\leq c_{0} \sup_{[0,t_{i}]}Q(t).
\end{equation*}

Let $Q_{i}=\sup_{[0,t_{i}]}Q(t)$ then there exist $x_{i}$, $\tilde{t}_{i}\rightarrow T$ such that $Q_{i}=|\Rm(x_{i},\tilde{t}_{i})|$. 
Now we can invoke a blow-up argument as in Theorem \ref{integralcond} around these points to obtain a singularity model $(M_{\infty},g_{\infty}(t),x_{\infty})$, $t\in [-\infty,0]$, with $|\Rm_{\infty}(x_{\infty},0)|=1.$ 

On the other hand, we have
\begin{align*}
\int_{-1}^{0}\int_{M}|\RR(g_{i}(s))|^{n/2+2}d\mu_{g_{i}(s)}ds&= \frac{1}{Q_{i}}\int_{\tilde{t}_{i}-\frac{1}{Q_{i}}}^{\tilde{t}_{i}}\int_{M}|\RR(g(t)|^{n/2+2}d\mu_{g(t)}dt\\
&\leq \frac{c_{0} Q_{i}-2c_{0}Q(0)}{a_{i}Q_{i}} \rightarrow 0.
\end{align*}
Thus, by the dominating convergence theorem, the limit solution is scalar flat, which is a contradiction to the uniform-growth condition.
\end{proof}

\begin{proof}({\bf Theorem \ref{mt2}})
Applying Lemma \ref{picnamtrick} with the function $\psi(s)=s\ln(1+s)^{p}$, $0\leq p\leq 1$ (it is easy to check that it is nondecreasing and $\int_{1}^{\infty}\frac{1}{\psi(s)}ds=\infty$) and Lemma \ref{picmvi} yields the result.
\end{proof}
\begin{remark} If $p=0$ we recover the first half of Theorem \ref{integralcond}.
\end{remark}

\def\cprime{$'$}
\bibliographystyle{plain}
\bibliography{bio}

\end{document}